\documentclass [a4paper, 12pt]{amsart}

\usepackage[margin=2.5cm]{geometry} 
\usepackage[alphabetic]{amsrefs}
\usepackage{wasysym,stmaryrd,enumerate}

\usepackage{amsthm}
\usepackage{latexsym}
\usepackage{times}
\usepackage{graphicx}
\usepackage{yhmath}

\usepackage{amsmath}
\usepackage{amsfonts}
\usepackage{amssymb}
\usepackage{amsthm}
\usepackage{graphicx}
\usepackage{pb-diagram}
\usepackage{epstopdf}
\usepackage{fancyhdr}
\usepackage[all]{xy}

\newtheorem{cor}{Corollary}[section]
\newtheorem{te}[cor]{Theorem}
\newtheorem{p}[cor]{Proposition}

\newtheorem{lemma}[cor]{Lemma}
\theoremstyle{definition}
\newtheorem{de}[cor]{Definition}
\theoremstyle{remark}

\newtheorem{nt}[cor]{Notation}

\newcommand{\cz}{\mathbb{C}}
\newcommand{\nz}{\mathbb{N}}

\newcommand{\bb}{\mathcal{B}}

\newcommand{\pp}{\mathcal{P}}

\newcommand{\vp}{\varphi}
\newcommand{\ve}{\varepsilon}
\newcommand{\es}{\emptyset}

   {\begin{verse}%
     \small }%
   {\end{verse}}

\DefineSimpleKey{bib}{arx}
\BibSpec{arx}{%
  +{}{\PrintAuthors} {author}
  +{,}{ \textit} {title}
  +{}{ \parenthesize} {date}
  +{,}{ } {note}
  +{.}{ } {transition}
}

\begin{document}

\title{Two special subgroups of the universal sofic group}

\author{Matteo Cavaleri}
\address[M. Cavaleri]{Institute of Mathematics of the Romanian Academy, 21 Calea Grivitei Street, 010702 Bucharest, Romania}
\email{matte.cavaleri@gmail.com}

\author{Radu B. Munteanu}
\address[R.B. Munteanu]{Institute of Mathematics of the Romanian Academy, 21 Calea Grivitei Street, 010702 Bucharest, Romania and Department of Mathematics, University of Bucharest, 14 Academiei Street}
\email{radu-bogdan.munteanu@g.unibuc.ro}

\author{Liviu P\u aunescu}
\address[L. P\u aunescu]{Institute of Mathematics of the Romanian Academy, 21 Calea Grivitei Street, 010702 Bucharest, Romania}
\email{liviu.paunescu@imar.ro}


\begin{abstract}
We define a subgroup of the universal sofic group, obtained as the normaliser of a separable abelian subalgebra. This subgroup can be obtained as an extension by the group of automorphisms on a standard probability space. We show that each sofic representation can be conjugated inside this subgroup. 
\end{abstract}
\maketitle

A well-known result due to Elek and Szabo is that a group is \emph{sofic} if and only if it is a subgroup of the \emph{universal sofic group} $\Pi_{k\to\omega}P_{n_k}$, \cite{El-Sz1}. Here $P_n\subset M_n(\cz)$ is the subgroup of permutation matrices, isomorphic to $Sym(n)$, the symmetric group. Elements of the universal sofic group act on $\Pi_{k\to\omega}D_{n_k}$, where $D_n\subset M_n(\cz)$ is the subalgebra of diagonal matrices. As an abelian type $II_1$ von Neumann algebra, $\Pi_{k\to\omega}D_{n_k}$ is isomorphic to the algebra of functions on a probability space, $L^\infty(X_\omega,\mu_\omega)$, where $(X_\omega,\mu_\omega)$ is the Loeb space.

The above picture has been fruitful in the study of soficity. It can be used to provide a proof of the fact that free product of sofic groups amalgamated over amenable subgroups is still sofic, \cite{Pa1}, Corollary 3.7. It was also successfully used to provide a compact proof for stability of the commutant in permutations with respect to the Hamming distance, \cite{Ar-Pa}. It is therefore natural to wish for a better understanding of these objects and their interaction.

We will not go through the basics of ultraproducts with respect to $\omega$, a non-principal ultrafilter on $\nz$, and $\{n_k\}_k\subset\nz$, an increasing sequence of natural numbers. The reader can consult the vast literature on the subject, including any of the previous cited articles or introductory papers as \cite{Pe} or \cite{Ca-Lu}.

Let $(X,\mu)$ be the unit interval endowed with the Lebesgue measure. In this paper we first construct a canonic embedding $L^\infty(X,\mu)\hookrightarrow L^\infty(X_\omega,\mu_\omega)$. In the second section we introduce the group $\mathcal{GA}$, the subgroup of the universal sofic group that normalises $L^\infty(X,\mu)$. In section 3, we prove that any sofic group is a subgroup of $\mathcal{GA}$. In the forth section we use the \emph{Coxeter length} to study this group, and in the last section we obtain $\mathcal{GA}$ as an extension of $Aut(X,\mu)$.

\section{The Loeb space}

We already defined the Loeb space $(X_\omega,\bb_\omega,\mu_\omega)$ by the equation $\Pi_{k\to\omega}D_{n_k}\simeq L^\infty(X_\omega,\mu_\omega)$, but we need a better understanding of its structure. For this we have to enter the realm of non-standard analysis in more depth than just considering some metric ultraproducts. The Loeb space was introduced in \cite{Lo}. We follow here methods from \cite{El-Sze}.

As a set, $X_\omega$ is the \emph{algebraic ultraproduct} of sets $\{1,2,\ldots,n_k\}$. From now on, $(x_k)_k$ or $(y_k)_k$ will denote elements in the Cartesian product $\Pi_k\{1,2,\ldots,n_k\}$. On this set we define the equivalence relation $(x_k)_k\sim(y_k)_k \Leftrightarrow \{k:x_k=y_k\}\in\omega$, i.e. two sequences are equivalent if they are equal on a subset in the ultrafilter. The algebraic ultraproduct is defined as: $X_\omega=\Pi_k\{1,2,\ldots,n_k\}/\sim$. Not to overload notations, we still denote by $(x_k)_k$ its class in $X_\omega$.

We now proceed to construct a measurable structure on $X_\omega$. Let $A_k\subset\{1,2,\ldots,n_k\}$ and construct $\Pi_{k\to\omega}A_k=\{{(x_k)_k}\in X_\omega:x_k\in A_k\}$. Let $\bb_\omega^0$ be the collection of all such subsets of $X_\omega$. Moreover, define $\mu_\omega:\bb_\omega^0\to[0,1]$ by $\mu_\omega(\Pi_{k\to\omega}A_k)=\lim_{k\to\omega}Card(A_k)/n_k$. Then $\bb_\omega^0$ is an algebra of sets and $\mu_\omega$ is a pre-measure. Due to Carathéodory's extension theorem, $\mu_\omega$ can be extended to $\bb_\omega^1$, the $\sigma$-algebra generated by $\bb_\omega^0$. Finally, we extend $\mu_\omega$ to $\bb_\omega$, the closure of $\bb_\omega^1$ under the measure $\mu_\omega$.

\subsection{The universal sofic action}

Let $p=\Pi_{k\to\omega}p_k\in\Pi_{k\to\omega}Sym(n_k)$. Then $p\big((x_k)_k\big)=\big(p_k(x_k)\big)_k$ defines an automorphism of $(X_\omega,\mu_\omega)$. Note the key role played by the measure $\mu_\omega$: the permutations $p_k$ are determined up to some error in the Hamming distance. This translates to the fact that the above automorphism of $X_\omega$ is well defined up to a set of $\mu_\omega$-measure $0$.

If $f\in L^\infty(X_\omega)$, we denote by $p(f)\in L^\infty(X_\omega)$ the function $p(f)(x)=f(p(x))$. If we identify $L^\infty(X_\omega)$ with $\Pi_{k\to\omega}D_{n_k}$ and $\Pi_{k\to\omega}Sym(n_k)$ with $\Pi_{k\to\omega}P_{n_k}$, both of which are subsets in $\Pi_{k\to\omega}M_{n_k}$, then $p(f)$ can be written as $p^{-1}\cdot f\cdot p$.

\subsection{The standard part.} The standard part function plays a key role in non-standard analysis. In our context it is defined as $St:X_\omega\to [0,1]$, $St\big({(x_k)_k}\big)=\lim_{k\to\omega}\frac{x_k}{n_k}$. It is a measure preserving function in the sense that $\mu_\omega(St^{-1}(A))=\mu(A)$, for every Lebesgue measurable set $A\subset[0,1]$, where $\mu$ is the Lebesgue measure, see Proposition \ref{canonic factoring}.  As a consequence we also have an inclusion of algebras $L^\infty([0,1],\mu)\hookrightarrow L^\infty(X_\omega,\mu_\omega)$, by identifying $\chi_A$ with $\chi_{St^{-1}(A)}$ (the characteristic function), as we now show.

\begin{de}
Define on $X_\omega$ the equivalence relation $x\sim  y\Leftrightarrow St( x)=St( y)$.
\end{de}

\begin{p}[Theorem 4.1 of \cite{Cu}]\label{canonic factoring}
The space with measure $X_\omega/\sim$, obtained by factoring $X_\omega$ by the above equivalence relation, is canonically isomorphic to a standard space $(X,\mu)$, where $X=[0,1]$ and $\mu$ is the Lebesgue measure.
\end{p}
\begin{proof}
Firstly, we show the equality for the Borel structure. Let $\bb^1$ be the $\sigma$-algebra on $[0,1]$ generated by intervals, aka Borel sets. Then $\bb$ is the closure of $\bb^1$ under $\mu$.

Let $\bb^2$ be the $\sigma$-algebra on $X$ induced by the map $St:X_\omega\to X$, i.e. $\bb^2=\{Y:St^{-1}(Y)\in\bb_\omega^1\}$. We want to show that $\bb^1=\bb^2$ and $\mu_\omega(St^{-1}(Y))=\mu(Y)$ for any such $Y$.

For $A\in\bb_\omega^0$ the reader can check that $St(A)\subset [0,1]$ is a closed set. As $\bb_\omega^0$ generates $\bb_\omega^1$, it follows that $St(A)\in\bb^1$ for any $A\in\bb_\omega^1$. As $Y=St(St^{-1}(Y))$, we get that $\bb^2\subset\bb^1$.

Now let $\lambda_1,\lambda_2\in [0,1]$, $\lambda_1<\lambda_2$. Let $\{x_k\}_k,\{y_k\}_k$ be sequences of natural numbers such that $\lim_{k\to\omega}\frac{x_k}{n_k}=\lambda_1$ and $\lim_{k\to\omega}\frac{y_k}{n_k}=\lambda_2$. Let $A=\Pi_{k\to\omega}\{x_k,x_k+1,\ldots,y_k\}\in\bb_\omega^0$. It can be checked that:
\[St^{-1}\big((\lambda_1,\lambda_2)\big)\subset A\subset St^{-1}\big([\lambda_1,\lambda_2]\big).\]
Both of these inequalities are strict, but this is not important to the argument. Now fix $\lambda\in (0,1]$ and let $\{\lambda_k\}_k$ be a strictly increasing sequence converging to $\lambda$, with $\lambda_0=0$. For every $j\in\nz$, let $\{x_k^j\}_k$ be a sequence such that $\lim_{k\to\omega}\frac{x_k^j}{n_k}=\lambda_j$, with $x_k^0=1$. Define $A_j=\Pi_{k\to\omega}\{x_k^j,\ldots,x_k^{j+1}-1\}\in\bb_\omega^0$. By the above inequalities:
\[St^{-1}\big([0,\lambda)\big)=\bigcup_jA_j.\]
This implies that $[0,\lambda)\in\bb^2$ for any $\lambda$. Moreover $\mu_\omega\left(St^{-1}([0,\lambda))\right)=\sum_j\mu_\omega(A_j)=\sum_j(\lambda_{j+1}-\lambda_j)=\lambda$.
So $\bb^1\subset\bb^2$ and $St:(X_\omega,\bb_\omega^1,\mu_\omega)\to (X,\bb^1,\mu)$ is measure preserving. The same is true for $St:(X_\omega,\bb_\omega,\mu_\omega)\to (X,\bb,\mu)$.
\end{proof}

\begin{nt}
We denote by $St^*\big(L^\infty(X,\mu)\big)$ the subalgebra $\{f\circ St:f\in L^\infty(X,\mu)\}\subset L^\infty(X_\omega,\mu_\omega)$.
\end{nt}

The last theorem shows that $X_\omega$ is a fibre bundle over $X$.

\subsection{The order relation}

\begin{de}
For $ x={(x_k)_k}, y={(y_k)_k}\in X_\omega$ define $ x\leqslant y$ if $\{k:x_k\leqslant y_k\}\in\omega$.
\end{de}

It can be checked that this is a total order relation, with antisymmetry following due to the algebraic ultraproduct construction ($ x= y$ iff $\{k:x_k=y_k\}\in\omega$). We now define \emph{initial segments}:

\begin{de}
For $x\in X_\omega$, denote by $I_x=\{ y\in X_\omega: y\leqslant x\}$. 
\end{de}

Note that $\mu_\omega(I_x)=St(x)$. Moreover:

\begin{p}\label{initial segments}
For $x,y\in X_\omega$, $I_x=I_y$ if and only if $St(x)=St(y)$. Also $I_x=\{y:St(y)<St(x)\}$.
\end{p}
\begin{proof}
It can be checked that $\mu_\omega(I_x\Delta I_y)=|St(x)-St(y)|$ for any $x,y\in X_\omega$. This implies the first statement. For the second part, notice that $\{y:St(y)<St(x)\}=St^{-1}\big([0,St(x))\big)$, so $\mu_\omega(\{y:St(y)<St(x)\})=St(x)=\mu_\omega(I_x)$. As $\{y:St(y)<St(x)\}\subset I_x$, the conclusion now follows.
\end{proof}

\section{Generalised maps}

We now proceed to the main definitions of this paper.

\begin{de}\label{GM}
Let $p\in\Pi_{k\to\omega}Sym(n_k)$. We say that $p|_X$ \emph{exists} if there is $\vp:X\to X$ such that $St(p(x))=\vp(St(x))$ for $\mu_\omega$-almost all $x\in X_\omega$. In this case we write $p|_X=\vp$.
\end{de}

\begin{de}
Denote by $\mathcal{GM}$ the set of elements $p\in\Pi_{k\to\omega}Sym(n_k)$ such that $p|_X$ exists, and by $\mathcal{GA}$ the set of such elements for which $p|_X$ is an automorphism of $(X,\mu)$.
\end{de}

\begin{p}
Let $p\in\Pi_{k\to\omega}Sym(n_k)$. If $p|_X=\vp$ exists, then $\vp$ is a measurable, measure-preserving map.
\end{p}
\begin{proof}
Let $A\subset X$ be a measurable set. The equality $\vp(St(x))=St(p(x))$ for $\mu_\omega$-almost all $x\in X_\omega$ implies $St^{-1}\big(\vp^{-1}(A)\big)=p^{-1}\big(St^{-1}(A)\big)$:
\begin{align*}
St^{-1}\big(\vp^{-1}(A)\big)=&\{x:St(x)\in\vp^{-1}(A)\}=\{x:\vp(St(x))\in A\}=\{x:St(p(x))\in A\}\\
=&\{x:p(x)\in St^{-1}(A)\}=\{x:x\in p^{-1}(St^{-1}(A))\}=p^{-1}\big(St^{-1}(A)\big).
\end{align*}
As $p^{-1}\big(St^{-1}(A)\big)\in\bb_\omega$ this equality implies $\vp^{-1}(A)$ is measurable in $X$. Moreover:
\[\mu(\vp^{-1}(A))=\mu_\omega(St^{-1}(\vp^{-1}(A)))=\mu_\omega(p^{-1}(St^{-1}(A)))= \mu_\omega(St^{-1}(A))=\mu(A).\]
We show later that every measure preserving map can be obtained as $p|_X$ for some $p$, Theorem \ref{all maps}.
\end{proof}

\begin{p}
Let $p\in\Pi_{k\to\omega}Sym(n_k)$. If $p$ and $p^{-1}$ are elements of $\mathcal{GM}$, then $p\in\mathcal{GA}$.
\end{p}
\begin{proof}
As $p$ and $p^{-1}$ are in $\mathcal{GM}$, there exist $\varphi: X\rightarrow X$ and $\psi: X\rightarrow X$ such that $p|_X=\varphi$ and $p^{-1}|_X=\psi$. 
Then, for almost all $x\in X_\omega$ we have:
\[\psi\circ \varphi(St(x))=\psi(St(p(x)))=St(p^{-1}(p(x)))=St(x).\]
It follows that $\psi\circ\vp=Id$. These are measure preserving maps, so $\vp$ is an automorphism of 
$(X,\mu)$ and therefore $p\in \mathcal{GA}$.
\end{proof}

Before proving that each measure preserving function is obtained as a $p|_X$, for a $p\in\Pi_{k\to\omega}Sym(n_k)$, we need the following well know lemma. It contains a basic principle of measure theory: by controlling the behaviour on sets, we control the behaviour on points a.e.

\begin{lemma}\label{sets to points}
Let $(Y,\nu)$ be a $\sigma$-finite measure space and $f,g: Y\rightarrow \mathbb{R}$ be measurable functions such that $f^{-1}(B)=g^{-1}(B)$ a.e. for every measurable set $B\subset \mathbb{R}$. Then $f(y)=g(y)$ for $\nu$-almost every $y\in Y$.  
\end{lemma}
\begin{proof}
Let us assume that $\nu(\{y\in Y | f(y)\neq g(y) \})>0$. If $\Delta=\{(x,x) :  x\in \mathbb{R}\}$ one can find closed-open intervals $I_n$, $J_n$, $n\geq 1$ such that $I_n\cap J_n=\emptyset$ and:
\[\mathbb{R}\times \mathbb{R}\setminus \Delta =\bigcup_{n=1}^\infty I_n\times J_n.\]
Then there exist a measurable set $A\subset Y$, $\nu(A)>0$ and two closed-open intervals  $I$ and $J$ with $I\cap J=\emptyset$ and  such that
\[\{(f(y),g(y) ):  y\in A \}\subset I\times J.\]
Thus $f(A)\subset I$ and $g(A)\subset J$, so $A\subseteq f^{-1}(I)$ and $A\subseteq g^{-1}(J)$. On the other hand: 
\[g^{-1}(I)\cap A\subseteq g^{-1}(I)\cap g^{-1}(J)=g^{-1}(I\cap J)=\es.\]
As $\nu(A)>0$, this implies that $f^{-1}(I)\neq g^{-1}(I)$ which is a contradiction. 
\end{proof}

\begin{te}\label{all maps}
Let $\vp:X\to X$ be a measure preserving map. Then there exists $p\in\Pi_{k\to\omega}Sym(n_k)$ such that $p|_X=\vp$.
\end{te}
\begin{proof}
We need to construct $p\in\Pi_{k\to\omega}Sym(n_k)$ such that $St^{-1}(\varphi^{-1}(A))=p^{-1}(St^{-1}(A))$ for any $A\subset X$. The proof is similar to the one in \cite{Pa1}, Proposition 3.3. Choose $\{A_i\}_i$ a finite partition of $X$. Then $\{St^{-1}(A_i)\}_i$ and $\{St^{-1}(\vp^{-1}(A_i))\}_i$ are two partitions of $X_\omega$ with $\mu_\omega(St^{-1}(A_i))=\mu_\omega(St^{-1}(\vp^{-1}(A_i)))$ for any $i$. By Lemma 3.2 in \cite{Pa1}, there exists $q_1\in\Pi_{k\to\omega}Sym(n_k)$ such that $q_1(St^{-1}(A_i))=
St^{-1}(\vp^{-1}(A_i))$ for any $i$. We construct a sequence of elements $q_j$ that work for finer and finer partitions of $X$. By a diagonal argument we construct $q$ such that 
$q(St^{-1}(A))=St^{-1}(\vp^{-1}(A))$ for any $A\subset X$. Define $p=q^{-1}$. By the previous lemma, applied for the functions $St\circ p,\vp\circ St:X_\omega\to X$ we get that $St(p(x))=\vp(St(x))$ for almost all $x\in X_\omega$.
\end{proof}

\subsection{Generalised automorphisms}

\begin{p}\label{restricted permutation}
For $p\in\Pi_{k\to\omega}Sym(n_k)$, $p|_X=Id$ if and only if $p(I_x)=I_x$ for ($\mu_\omega$-almost) all $x\in X_\omega$.
\end{p}
\begin{proof}
Without any hypothesis on $p\in\Pi_{k\to\omega}Sym(n_k)$, we have:
\begin{align*}
\int_{ X_\omega}\mu_\omega(p(I_x)\Delta I_x)d\mu_\omega(x)=&\int_{ X_\omega}\mu_\omega(\{y:y\leqslant x\mbox{ and }p^{-1}(y)>x\mbox{ or }y>x\mbox{ and }p^{-1}(y)\leqslant x\})d\mu_\omega(x)\\
=&\int_{X_\omega}\mu_\omega(\{x:y\leqslant x\mbox{ and }p^{-1}(y)>x\mbox{ or }y>x\mbox{ and }p^{-1}(y)\leqslant x\})d\mu_\omega(y).
\end{align*}
One can check that, for $a,b\in[0,1]$, $\mu(x:a\leqslant x\mbox{ and }b>x\mbox{ or }a>x\mbox{ and }b\leqslant x\})=|a-b|$. In the Loeb measure this translates to $\mu_\omega(\{x:y\leqslant x\mbox{ and }p^{-1}(y)>x\mbox{ or }y>x\mbox{ and }p^{-1}(y)\leqslant x\})=|St(y)-St(p^{-1}(y))|$. We reach the following equation, interesting in its own right:
\[\int_{ X_\omega}\mu_\omega(p(I_x)\Delta I_x)d\mu_\omega(x)=\int_{ X_\omega}|St(y)-St(p^{-1}(y))|d\mu_\omega(y).\]

Assume now that $p(I_x)=I_x$ almost everywhere. Then $\int_{x\in X_\omega}\mu_\omega(p(I_x)\Delta I_x)d\mu_\omega(x)=0$. 
It follows that $St(y)=St(p^{-1}(y))$ for $\mu_\omega$-almost all $y\in X_\omega$ and hence the conclusion.

For the reverse implication, assume that $St(p(x))=St(x)$ almost everywhere. By Proposition \ref{initial segments}, $p(I_x)=\{p(y):St(y)<St(x)\}=\{p(y):St(p(y))<St(x)\}\subset I_x$. Hence $\mu_\omega(p(I_x)\setminus I_x)=0$ for any $x\in X_\omega$. As these sets are of equal measure, they must be equal.
\end{proof}

Rebemeber that $St^*\big(L^\infty(X,\mu)\big)=\{f\circ St:f\in L^\infty(X,\mu)\}\subset L^\infty(X_\omega,\mu_\omega)$.

\begin{te}\label{normaliser characterisation}
Let $p\in\Pi_{k\to\omega}Sym(n_k)$. Then $p\in\mathcal{GA}$ if and only if $p\big(St^{*}(L^\infty(X))\big)=St^{*}(L^\infty(X))$, i.e. $p$ is in the normaliser of $St^{*}(L^\infty(X))$ when ultraproducts are identified with subsets in $\Pi_{k\to\omega}M_{n_k}$.
\end{te}
\begin{proof}Let $p$ be in the normalizer of $St^{*}(L^\infty(X,\mu))$.  For $f\in L^\infty(X,\mu)$ we define $\Phi(f)$ to be the unique function in $L^\infty(X,\mu)$ such that $p( f\circ St)=\Phi(f)\circ St$. Hence $f\rightarrow \Phi(f)$ defines an automorphism of $L^\infty(X,\mu)$ and then, there exists a nonsingular automorphism $\varphi$ of $(X,\mu)$ such that $\Phi(f)=f\circ \varphi$ for all $f\in L^\infty(X,\mu)$. Thus:
\[p(f\circ St)=f\circ\vp\circ St,\]
for all $f\in L^\infty(X,\mu)$. It follows that for $\mu_\omega$-almost all $x\in X_\omega$ we have:
\[p(f\circ St)(x)=f\circ St(p(x))=f\circ\vp(St(x)).\]
By Lemma \ref{sets to points}, it follows that:
\[St(p(x))=\vp(St(x)).\]
By Definition \ref{GM}, this translates to $p|_X=\vp$. As $\vp\in Aut(X,\mu)$, we get $p\in\mathcal{GA}$.

Conversely, if there exists $\varphi$ an automorphism of $(X,\mu)$ such that $p|_X=\varphi$ then: 
\[p(f\circ St)(x)=f\circ St(p(x))=(f\circ\vp)\circ St(x)\in St^*(L^\infty(X,\mu)),\]
for all $f\in L^\infty(X,\mu)$. This implies $p\big(St^{*}(L^\infty(X))\big)\subset St^{*}(L^\infty(X))$ and, by replacing $f$ with $f\circ\vp^{-1}$ in the above equality, we get the reverse inclusion. 
\end{proof}

\section{Sofic representations} The purpose of this section is to prove that any sofic representation of any group $\theta:G\to\Pi_{k\to\omega}Sym(n_k)$ can be conjugated inside $\mathcal{GA}$, i.e. there exists $p\in\Pi_{k\to\omega}Sym(n_k)$ such that $p\theta p^*\subset\mathcal{GA}$. Recall that a \emph{sofic representation} is a group morphism $\theta:G\to\Pi_{k\to\omega}Sym(n_k)$ such that $\ell_H(\theta(g))=1$ for any $g\neq e$, where $\ell_H$ is the normalised Hamming length.

We do this with the help of a theorem by Elek and Lippner, \cite{El-Li}, saying that a Bernoulli shift action of a sofic group is sofic. Ozawa has a nice proof of this result (\cite{Oz}, also Theorem 3.5 of \cite{Pa1}), but it uses an amplification, that would provide a weaker result in our context. This is why we need to inspect the original proof of Elek and Lippner.

In the following theorem, the space $Y=\{0,1\}^G$ is the product space of $\{0,1\}$ endowed with the normalised cardinal measure, indexed by the countable group $G$. We denote by $\beta:G\to Aut(Y)$ the Bernoulli shift action.

\begin{te}[Proposition 7.1 of \cite{El-Li}]\label{Elek-Lippner}
Let $\theta:G\to\Pi_{k\to\omega}P_{n_k}$ be a sofic representation. There exists $\overline\theta:L^\infty(Y)\rtimes_\beta G \to\Pi_{k\to\omega}M_{n_k}$ an embedding of the crossed product that extends $\theta$, such that $\overline\theta(L^\infty(Y))\subset\Pi_{k\to\omega}D_{n_k}$.
\end{te}
\begin{proof}
For the reader's convenience, we outline here the main ideas in the original proof, adapted to our language and notation. The $\sigma$-algebra on $Y=\{0,1\}^G=\{f:G\to\{0,1\}\}$ is generated by the cylinder sets:
\[c_{g_1,g_2,\ldots,g_m}^{i_1,i_2,\ldots,i_m}=\{f\in Y:f(g_j)=i_j,\ j=1,\ldots,m\},\]
where $g_1,\ldots,g_m$ are distinct elements of $G$. The measure of such a cylinder is $1/2^m$. We denote by $Q_{g_1,g_2,\ldots,g_m}^{i_1,i_2,\ldots,i_m}\in L^\infty(Y)$ the projection onto this set.

The key observation is that we only need to construct $\overline\theta(Q_e^1)$, as the rest of the embedding is generated by the following relations:
\begin{align}
Q_g^1&=u_gQ_e^1u_g^*\\
Q_g^0&=Id-Q_g^1\\
Q_{g_1,g_2,\ldots,g_m}^{i_1,i_2,\ldots,i_m}&=Q_{g_1}^{i_1}\cdot Q_{g_2}^{i_2}\cdot \ldots\cdot Q_{g_m}^{i_m}.
\end{align}
These relations are written in $L^\infty(Y)\rtimes_\beta G$ and $u_g$ is the unitary corresponding to $g\in G$. All we need is to construct $a\in\Pi_{k\to\omega}D_{n_k}$, such that $Tr\big(\theta(g_1)a\theta(g_1)^*\cdot\theta(g_2)a\theta(g_2)^*\cdot\ldots\cdot\theta(g_m)a\theta(g_m)^*\big)=1/2^m$ for each $m$ and $g_1,\ldots,g_m\in G$.

We use the \emph{second moment method}: we consider the set of all projections $a_k\in\pp(D_{n_k})$, we compute the expected value of these traces, i.e. the average, and show that the variance, i.e. the deviation from the average, is sufficiently small. 

Let us first exemplify in the case of one projection. For now, we fix $n\in\nz$, dropping the $n_k$ index. The cardinality of $\pp(D_{n})$ is $2^{n}$, as each of the $n$ diagonal entries can independently be $0$ or $1$. We identify $a\in\pp(D_n)$ with this function $d_a:\{1,\ldots,n\}\to\{0,1\}$, representing the diagonal entries. Then $Tr(a)=\frac1n\sum_{x\in\{1,\ldots,n\}}d_a(x)$.
Moreover, for each $x$,  $d_a(x)=1$ for exactly half of the matrices $a\in\pp(D_n)$, i.e. $\frac1{2^n}\sum_{a\in\pp( D_{n})}d_a(x)=\frac12$. It follows that:
\[\frac1{2^{n}}\sum_{a\in\pp( D_{n})}Tr(a)=\frac1{2^{n}}\sum_{a\in\pp( D_{n})}\frac1n\sum_{x\in\{1,\ldots,n\}}d_a(x)=\frac1{n2^n}\sum_{x}\sum_{a}d_a(x)=\frac1n\sum_x\frac12=\frac12.\]
This doesn't mean that we can find $a\in\pp(D_n)$ such that $Tr(a)=1/2$ or close to this value. This is why we also compute the variance: 
\[\frac1{2^{n}}\sum_{a\in\pp( D_{n})}\big(Tr(a)-\frac12\big)^2=\frac1{2^{n}}\sum_{a\in\pp( D_{n})}Tr(a)^2-\frac1{2^{n}}\sum_{a\in\pp( D_{n})}Tr(a)+\frac14=\frac1{2^{n}}\sum_{a\in\pp( D_{n})}Tr(a)^2-\frac14.\]
If this value is small enough, we can interfere the existence of many projections with trace close to $1/2$. In order to compute $\sum Tr(a)^2$, we notice that $Tr(a)^2=Tr(a\otimes a)$. The function associated to $a\otimes a$ is $d_{a\otimes a}:\{1,\ldots,n\}\times\{1,\ldots,n\}\to\{0,1\}$, defined as $d_{a\otimes a}(x,y)=d_a(x)\cdot d_a(y)$. Then $Tr(a\otimes a)=\frac1{n^2}\sum_{x,y}d_{a\otimes a}(x,y)$. As before, we fix $(x,y)\in\{1,\ldots,n\}^2$, and estimate $\frac1{2^n}\sum_{a\in\pp(D_n)}d_{a\otimes a}(x,y)$. If $x=y$, then this sum is $\frac12$ as before. If $x\neq y$, then $d_{a\otimes a}(x,y)=1$ iff $d_a(x)=d_a(y)=1$. Only $1/4$ of elements of $\pp(D_n)$ satisfies this condition. All in all:
\[\frac1{2^{n}}\sum_{a\in\pp( D_{n})}Tr(a)^2=\frac1{n^2}\big(n\cdot\frac12+(n^2-n)\frac14\big)=\frac14+\frac1{4n}.\]
We proved that $\frac1{2^{n}}\sum\big(Tr(a)-\frac12\big)^2=\frac1{4n}$. It means that, for any $\lambda>0$, $(Tr(a)-\frac12\big)^2\geqslant\frac\lambda{4n}$ for at most $\frac{2^{n}}\lambda$ elements of $\pp(D_n)$ (in this setting, this is Chebyshev's inequality). As, we can increase $n$ arbitrary large, we can choose a dimension where $Tr(a)$ is close to $1/2$ for any proportion of projections we want. 

The proof of the theorem is not more difficult. In order to construct the required embedding $\overline\theta$, fix $\ve>0$, and $F\subset G$ a finite subset with $m$ its cardinality. We want to find a sufficiently large $n_k$ and $a\in D_{n_k}$, such that not only that $|Tr(a)-1/2|<\ve$, but also $|Tr(p_{1}a^{s_1}p_{1}^*\cdot\ldots\cdot p_{m}a^{s_m}p_{m}^*)-1/2^m|<\ve$, where $p_{1},\ldots,p_m$ is an enumeration of the set $\{\theta_k(g):g\in F\}$ and $s_j\in\{0,1\}$ with $a^1=a$ and $a^0=1-a$. Denote by $N$ the number of these inequalities ($N=2^m$).

We choose $\lambda=N+1$ and show, using the above method, that each of these conditions fails for at most  $2^{n_k}/\lambda$ projections $a\in D_{n_k}$. This implies the existence of a projection that simultaneously satisfies all those inequalities. We identify a permutation matrix $p\in P_{n_k}$ with an element of $Sym(n_k)$, meaning a function $\{1,\ldots,n_k\}\to\{1,\ldots,n_k\}$. In order to greatly simplify the writing we assume that for all entires $x\in\{1,\ldots,n_k\}$ the permutations $p_1,\ldots,p_m$ take different values, i.e. the set $\{p_1(x),\ldots,p_m(x)\}$ has cardinality $m$ for each $x$. This is true for most entries in a sofic representation. The complete proof will separate the set $\{1,\ldots,n_k\}$ into these ``good" points and ``bad" points, the error that the definition of sofic groups allows. With this assumption, we have:
\[\frac1{2^{n_k}}\sum_{a\in\pp( D_{n_k})}Tr(p_{1}a^{s_1}p_{1}^*\cdot\ldots\cdot p_{m}a^{s_m}p_{m}^*)= \frac1{2^m}.\]
As before, this computation is done by fixing an entry $x\in\{1,\ldots,n_k\}$ and count how often $b=p_{1}a^{s_1}p_{1}^*\cdot\ldots\cdot p_{m}a^{s_m}p_{m}^*$ is $1$ on this position. Note that $d_b(x)=1$ iff $d_{p_ja^{s_j}p_j^*}(x)=1$ for each $j=1,\ldots,m$. Also $d_{p_ja^{s_j}p_j^*}(x)=d_{a^{s_j}}(p_j(x))$. In the end $d_b(x)=1$ iff $d_a(p_j(x))=s_j$ for all $j=1,\ldots,m$. Here we use the fact that numbers $\{p_j(x)\}_j$ are distinct, so that indeed $\frac1{2^{n_k}}\sum_ad_{p_{1}a^{s_1}p_{1}^*\cdot\ldots\cdot p_{m}a^{s_m}p_{m}^*}(x)=\frac1{2^m}$. 

The more difficult part is to compute the variance:
\[\frac1{2^{n_k}}\sum_{a\in\pp( D_{n_k})}\big(Tr(p_{1}a^{s_1}p_{1}^*\cdot\ldots\cdot p_{m}a^{s_m}p_{m}^*)-\frac1{2^m}\big)^2=\frac1{2^{n_k}}\sum_{a\in\pp( D_{n_k})}Tr(p_{1}a^{s_1}p_{1}^*\cdot\ldots\cdot p_{m}a^{s_m}p_{m}^*)^2-\frac1{4^m}.\]

Using the same notation $b=p_{1}a^{s_1}p_{1}^*\cdot\ldots\cdot p_{m}a^{s_m}p_{m}^*$, we have that $Tr(b)^2=Tr(b\otimes b)$. Then $d_{b\otimes b}(x,y)=1$ iff $d_a(p_j(x))=d_a(p_j(y))=s_j$ for all $j=1,\ldots,m$. If $\{p_j(x):j=1,\ldots,m\}$ and $\{p_j(y):j=1,\ldots,m\}$ are disjoint sets, then these conditions hold for exactly $2^{n_k-2m}$ projections $a\in\pp(D_{n_k})$. If those two sets intersect, we are not interested in computing the number of good projections. It can be zero, or $2^{n_k-m}$ if $x=y$. Thus, we need to count the number of pairs $(x,y)$ such that $\{p_j(x):j=1,\ldots,m\}$ and $\{p_j(y):j=1,\ldots,m\}$ are disjoint. Assume they are not. Then there is $j_1,j_2\in\{1,\ldots,m\}$ such that $p_{j_1}(x)=p_{j_2}(y)$, or $y=p_{j_2}^{-1}p_{j_1}(x)$. In total we get $m^2$ forbidden values of $y$ for any fixed $x$. Hence, the number of good $(x,y)$ pairs is $n_k(n_k-m^2)$, which is a generalisation of our previous $n(n-1)$ (actually the number of good pairs is slightly less, as we have to exclude also those for which $\{p_j(x)\}_j$ or $\{p_j(y)\}_j$ are not collection of distinct numbers). We can now provide an estimate.
\[\sum_{a\in\pp( D_{n_k})}Tr(p_{1}a^{s_1}p_{1}^*\cdot\ldots\cdot p_{m}a^{s_m}p_{m}^*)^2\leqslant \frac1{n_k^2}\big(n_k(n_k-m^2)\frac1{4^m}+n_km^2\frac1{2^m})=\frac1{4^m}+\frac{c_m}{n_k},\]
where $c_m$ is a constant on $m$. Now we know that $\big(Tr(p_{1}a^{s_1}p_{1}^*\cdot\ldots\cdot p_{m}a^{s_m}p_{m}^*)-\frac1{2^m}\big)^2<\frac{\lambda c_m}{n_k}$ for all but at most $\frac{2^{n_k}}{\lambda}$ elements in $\pp(D_{n_k})$. We choose a sufficiently large $n_k$ such that $\frac{\lambda c_m}{n_k}<\ve^2$ and we are done.
\end{proof}

All this effort just to get rid of an amplification. Before proving the result of this section, we cite the following proposition.

\begin{p}[Proposition 3.3 of \cite{Pa1}]
Let $\theta_1$, $\theta_2$ be two embeddings of $L^\infty(X,\mu)$ in $\Pi_{k\to\omega}D_{n_k}$. Then there exists $p\in\Pi_{k\to\omega}P_{n_k}$ such that $\theta_2=p\theta_1p^*$.
\end{p}

\begin{te}
Let $\theta:G\to\Pi_{k\to\omega}Sym(n_k)$ be a sofic representation. Then there exists $p\in\Pi_{k\to\omega}Sym(n_k)$ such that $p\theta p^*\subset\mathcal{GA}$.
\end{te}
\begin{proof}
By Theorem \ref{Elek-Lippner}, we have an extension $\overline\theta:L^\infty(Y)\rtimes_\beta G \to\Pi_{k\to\omega}M_{n_k}$. By the above proposition there is $p\in\Pi_{k\to\omega}P_{n_k}$ such that $St^{*}(L^\infty(X))=p\overline\theta(L^\infty(Y))p^{-1}$. For any $g\in G$, $p\theta(g) p^*$ is acting on this abelian subalgebra. By Theorem \ref{normaliser characterisation}, $p\theta(g) p^*$ is in $\mathcal{GA}$.
\end{proof}

This theorem shows that, when investigating sofic groups, we can restrict our study from the universal sofic group to the subgroup of generalised automorphisms.

\section{The Coxeter semi-length}

In this section we investigate the connection between generalised maps and the order relation on $X_\omega$. This study leads us to the Coxeter length, that we now define.

\begin{de}
For $p\in Sym(n)$ the \emph{Coxeter length} is defined as:
\[\ell_C(p)=\frac2{n(n-1)}Card\{i<j:p(i)>p(j)\}.\]
Only the identity has Coxeter length equal to zero. If $p(i)=n+1-i$, then $\ell_C(p)=1$, independently of $n$. The factor $\frac2{n(n-1)}$ from the definition plays the role of normalising the length.
\end{de}

\begin{p}
For $p\in Sym(n)$, $\ell_C(p)\leqslant 2\cdot \ell_H(p)$.
\end{p}
\begin{proof}
Let us assume that $\ell_H(p)=\frac{k}{n}$ with $k\geq 0$. Since $p$ has $n-k$ fixed points, the number of pairs $(i,j)$ such that $i<j$ and $p(i)<p(j)$ is at least ${{n}\choose{n-k}}=\frac{1}{2}(n-k)(n-k-1)$. Therefore: 
\[\ell_C(p)\leq 1-\frac{(n-k)(n-k-1)}{n(n-1)}=\frac{2kn-k(k+1)}{n(n-1)}\leq \frac{2kn-2k}{n(n-1)}=\frac{2k}{n}=2\cdot \ell_H(p).\]
\end{proof}

\begin{cor}
The function $\ell_C:\Pi_{k\to\omega}Sym(n_k)\to [0,1]$, defined as $\ell_C(\Pi_{k\to\omega}p_k)=\lim_{k\to\omega}\ell_C(p_k)$, is a well defined semi-length on the universal sofic group.
\end{cor}

The following proposition provides a nice characterisation of elements of the universal sofic group that act trivially on $(X,\mu)$.

\begin{p}
Let $p=\Pi_{k\to\omega}p_k\in\Pi_{k\to\omega}Sym(n_k)$. Then $p|_X=Id$ if and only if $\ell_C(p)=0$.
\end{p}
\begin{proof}
We define $Inv(p)=\{(x,y)\in X_\omega^2: \,x\leqslant y, \ p(x)>p(y)\}\cup \{(x,y)\in X_\omega^2: \ x>y, \,p(x)\leq p(y)\}$. This set can be described as $Inv(p)=\{(x,y)\in X_\omega^2: \  x\in I_y\Delta p^{-1}I_{p(y)}\}$. By Fubini's theorem, $\mu_\omega\times\mu_\omega(Inv(p))=\int_{X_\omega}\mu_\omega(I_y\Delta p^{-1}I_{p(y)})d\mu_\omega(y)$. Moreover:
\[\mu_\omega\times\mu_\omega(Inv(p))=
\lim_\omega \frac{2Card(\{(x_k,y_k)\in\{1,\ldots,n_k\}^2:\, x_k\leqslant y_k,\ p_k(x_k)>p_k(y_k)\})}{{n_k}^2}=\ell_C(p).\]
It follows that $\ell_C(p)=0$ iff $\int_{y\in X_\omega}\mu_\omega(I_y\Delta p^{-1}I_{p(y)})d\mu_\omega=0$ iff $I_y=p^{-1}I_{p(y)}$ for $\mu_\omega$-almost all $y\in X_\omega$.

If $I_y=p^{-1}I_{p(y)}$ then $\mu_\omega(I_y)=\mu_\omega(I_{p(y)})$ so $St(y)=St(p(y))$ almost everywhere. This is the definition of $p|_X=Id$.

Assume now that $p|_X=Id$. By Proposition \ref{initial segments} $I_y=I_{p(y)}$, and by Propostion \ref{restricted permutation} $I_{p(y)}=p^{-1}I_{p(y)}$. Hence $I_y=p^{-1}I_{p(y)}$ and thus $\ell_C(p)=0$.
\end{proof}

\begin{de}
Denote by $\ell_0=\{p\in\Pi_{k\to\omega}Sym(n_k):\ell_C(p)=0\}$.
\end{de}

In the next section we shall see that the group of generalised automorphisms is an extension of $\ell_0$ by $Aut(X,\mu)$.

By using a result of Diaconis and Graham, we can link the group of $\ell_0$ with the notion of \emph{total displacement}, with no additional effort.

\begin{de}
For $p\in Sym(n)$,  the \emph{normalized total displacement} is defined as: 
\[T(p):=\frac{2}{n(n-1)}\sum^n_{i=1} |p(i)-i|.\]
\end{de}

The name \emph{total displacement} comes from \cite{Kn}; however Diaconis and Graham already showed the relation between Coxeter distance and total displacement.

\begin{p}[Theorem 2 in \cite{Di-Gr}]
For any $p\in Sym(n)$ we have $\ell_C(p)\leq T(p)\leq 2\ell_C(p)$.
\end{p}

It follows that the normalised total displacement can be defined also for elements of the universal sofic group, as an ultralimit.

\begin{cor}
Let $p=\Pi_{k\to\omega}p_k\in\Pi_{k\to\omega}Sym(n_k)$. Then $ p\in\ell_0\Leftrightarrow \lim_{k\to\omega}T(p_k)=0$.
\end{cor}

\section{A short exact sequence}

\begin{te}
The following is a short exact sequence: $0\to\ell_0\to \mathcal{GA}\to Aut(X,\mu)\to 0$.
\end{te}
\begin{proof}
Define $\Psi:\mathcal{GA}\to Aut(X,\mu)$ by $\Psi(p)=\vp$, where $p|_X=\vp$. Then $\Psi$ is a morphism and, by Proposition \ref{all maps}, it is surjective. By definition $Ker(\Psi)=\ell_0$. This proves the statement.
\end{proof}

Settling the type of this extension seems challenging. We can at least prove that it is not trivial, i.e. $\mathcal{GA}$ is not obtained as a direct product via the maps contained in the short exact sequence.

\begin{p}
The commutant of $\ell_0$ in the universal sofic group is trivial.
\end{p}
\begin{proof}
Let $p=\Pi_{k\to\omega}p_k\in\Pi_{k\to\omega}Sym(n_k)$ be an element commuting with $\ell_0$. Let $s_k^1,s_k^2\in Sym(n_k)$ be defined as follows: $s_k^1(i)=i+1$ with the exception of $s_k^1(n_k)=1$ and $s_k^2(2i)=2i$, $s_k^2(2i+1)=2i+3$, again with the exception of the largest odd number smaller than $n_k$, for which $s_k^2$ is equal to $1$. It is easy to see that $s_1=\Pi_{k\to\omega}s_k^1$ and $s_2=\Pi_{k\to\omega}s_k^2$ are elements of $\ell_0$. It follows that $ps_1=s_1p$ and $ps_2=s_2p$.

Let $A_k$ be the set of points $i\in\{1,\ldots,n_k\}$ such that $p_ks_k^1(i)=s_k^1p_k(i)$ and $p_ks_k^2(i)=s_k^2p_k(i)$. We know that $Card(A_k)/n_k\to_{k\to\omega}1$. Then, for $i\in A_k$, the first condition implies that $p_k(i+1)=p_k(i)+1$ and the second one amounts to $i$ and $p_k(i)$ having the same parity.

Let $B_k$ be the set of non-fixed even points of $p_k$, i.e. $B_k=\{2i:p_k(2i)\neq 2i\}$. Let $C_k$ be a maximal subset of $B_k$ with the property that $p_k(C_k)\cap C_k=\es$. Then $Card(C_k)\geqslant Card(B_k)/3$ (if $x\in B_k\setminus C_k$ cannot be added to $C_k$ it means that either $p_k(x)\in C_k$ or $x\in p_k(C_k)$).

Construct $s_k^3$ as follows: $s_k^3(2i)=2i+1$ and $s_k^3(2i+1)=2i$ for any $i$ such that $2i\in C_k$, and $s_k^3(i)=i$ otherwise. It is easy to see that $s_3=\Pi_{k\to\omega}s_k^3$ is in $\ell_0$, so $ps_3=s_3p$. However, for $2i\in A_k\cap C_k$, we have:
\begin{align*}
p_ks_k^3(2i)=&p_k(2i+1)=p_k(2i)+1\\
s_k^3p_k(2i)=&p_k(2i)\mbox{ because } p_k(2i)\notin C_k\mbox{ and it is even.}
\end{align*}
It follows that $d_H(p_ks_k^3,s_k^3p_k)\geqslant Card(A_k\cap C_k)/n_k$. As $ps_3=s_3p$, we get $Card(A_k\cap C_k)/n_k$ converges in the ultralimit to $0$. As $Card(A_k)/n_k\to_{k\to\omega}1$, it must be that $Card(C_k)/n_k\to_{k\to\omega}0$. Then $Card(B_k)/n_k\to_{k\to\omega}0$, so almost all even points in $\{1,\ldots,n_k\}$ are fixed points for $p_k$. A similar argument can be performed for odd points.

\end{proof}

\begin{cor}
The extension $0\to\ell_0\to \mathcal{GA}\to Aut(X,\mu)\to 0$ is not trivial.
\end{cor}

\section*{Acknowledgements}

This work was supported by a grant of the Romanian National Authority for Scientific Research and Innovation, CNCS - UEFISCDI, project number PN-II-RU-TE-2014-4-0669.

Parts of this article were inspired by Gabor Elek's mini-course on sofic groups at the Erwin Schr\"{o}dinger Institute for Mathematics and Physics under the Measured Group Theory program in 2016.

\end{document}